\documentclass[11pt,reqno]{amsproc}
\usepackage{amsmath,amsfonts,amssymb,amsthm}
\usepackage[abbrev,lite,nobysame]{amsrefs}
\usepackage{mathrsfs}
\usepackage{graphics,graphicx}
\usepackage[usenames,dvipsnames]{color}
\usepackage{times}
\usepackage{bbm}
\usepackage[margin=1in]{geometry}
\usepackage[colorlinks=true, pdfstartview=FitV, linkcolor=blue, citecolor=blue, urlcolor=blue]{hyperref}

\makeatletter

\renewcommand\subsection{\@startsection{subsection}{2}%
\normalparindent{.5\linespacing\@plus.7\linespacing}{-.5em}
{\normalfont\bfseries}}

\renewcommand\subsubsection{\@startsection{subsubsection}{3}%
\normalparindent{.5\linespacing\@plus.7\linespacing}{-.5em}
{\normalfont\bfseries}}

\def\@tocline#1#2#3#4#5#6#7{\relax
  \ifnum #1>\c@tocdepth % then omit
  \else
    \par \addpenalty\@secpenalty\addvspace{#2}%
    \begingroup \hyphenpenalty\@M
    \@ifempty{#4}{%
      \@tempdima\csname r@tocindent\number#1\endcsname\relax
    }{%
      \@tempdima#4\relax
    }%
    \parindent\z@ \leftskip#3\relax \advance\leftskip\@tempdima\relax
    \rightskip\@pnumwidth plus4em \parfillskip-\@pnumwidth
    #5\leavevmode\hskip-\@tempdima
      \ifcase #1
       \or\or \hskip 1em \or \hskip 2em \else \hskip 3em \fi%
      #6\nobreak\relax
    \dotfill\hbox to\@pnumwidth{\@tocpagenum{#7}}\par
    \nobreak
    \endgroup
  \fi}
\makeatother

\newtheorem{theorem}{Theorem}
\newtheorem{proposition}{Proposition}[section]

\theoremstyle{definition}

\newtheorem{remark}[proposition]{Remark}

\numberwithin{equation}{section}

\newcommand\e{{\rm e}}
\newcommand\dd{{\rm d}}
\newcommand\ddt{{\frac{\dd}{\dd t}}}
\def\Re{{\rm Re}}

\def\l {\langle}
\def\r {\rangle}
\newcommand\de{{\partial}}

\newcommand{\norm}[1]{\left\lVert #1 \right\rVert}

\newcommand{\ZZ}{\mathbb{Z}}
\newcommand\TT {{\mathbb T}}
\newcommand\RR {{\mathbb R}}

\newcommand\sfE{{\mathsf E}}

\newcommand\sfa{{\mathsf a}}
\newcommand\sfb{{\mathsf b}}
\newcommand\sfc{{\mathsf c}}

\newcommand\sfm{{\mathsf m}}

\usepackage{mathtools}
\mathtoolsset{showonlyrefs=true}

\begin{document}
\title[Suppression of lift-up effect in 3D Boussinesq]{Suppression of lift-up effect in the 3D Boussinesq equations \\ around a stably stratified Couette flow} 

\author[M. Coti Zelati]{Michele Coti Zelati}
\address{Department of Mathematics, Imperial College London, London, SW7 2AZ, UK}
\email{m.coti-zelati@imperial.ac.uk}

\author[A. Del Zotto]{Augusto Del Zotto}
\address{Department of Mathematics, Imperial College London, London, SW7 2AZ, UK}
\email{a.del-zotto20@imperial.ac.uk}



\begin{abstract}
In this paper, we establish linear enhanced dissipation results for the three-dimensional Boussinesq equations around a stably stratified Couette flow, in the viscous and thermally diffusive setting. The dissipation rates are faster compared to those observed in the homogeneous Navier-Stokes equations, in light of the interplay between velocity and temperature, driven by buoyant forces.

Our approach involves introducing a change of variables grounded in a Fourier space symmetrization framework. This change elucidates the energy structure inherent in the system. Specifically, we handle non-streaks modes through an augmented energy functional, while streaks modes are amenable to explicit solutions. This explicit treatment reveals the oscillatory nature of shear modes, providing the elimination of the well-known three-dimensional instability mechanism known as the ``lift-up effect''.
\end{abstract}

\maketitle

\newcommand\q{q}
\newcommand\Q{Q}
\newcommand\GG{G}
\setcounter{tocdepth}{3}
\tableofcontents

\section{Introduction}
Thermal convection, a phenomenon ubiquitous in nature and engineering, plays a fundamental role in shaping the behavior of fluids under temperature gradients. Understanding and modeling this intricate interplay between fluid motion and heat transfer is of paramount importance in various scientific disciplines and technological applications. At the heart of the mathematical description of thermal convection lie the viscous and thermally diffusive Boussinesq equations
\begin{equation}\label{eq:3DBoussinesq}
\begin{cases}
\de_t V+(V\cdot\nabla)V+\nabla P=\nu\Delta V+\mathfrak{g}\Theta \hat y,\\
\de_t\Theta+V\cdot\nabla\Theta=\kappa\Delta \Theta,\\
\nabla\cdot V=0.
\end{cases}
\end{equation}
Above, $V=(V^1,V^2,V^3)$ is the velocity of an incompressible fluid with pressure $P$ and viscosity $\nu>0$, and $\Theta$ is its temperature, with thermal diffusivity $\kappa>0$. The equations for $V$ and $\Theta$ are coupled through buoyancy acting in the second component (as $\hat{y}=(0,1,0)$), with gravity constant $\mathfrak{g}>0$.

\subsection{The stably stratified Couette flow}
We restrict ourselves to the spatial setting $(x,y,z)\in \TT\times\RR\times\TT$, and study the linearized dynamics of solutions around the steady state
\begin{equation}\label{eq:steadystate}
U^s=(y,0,0),\qquad \de_y P^s=\mathfrak{g}(1+\alpha y), \qquad \Theta^s=1+\alpha y,\qquad \alpha>0,
\end{equation}
namely, a stably stratified Couette flow. Indeed, since $\alpha>0$ this solution represents the scenario in which the fluid is flowing along the Couette flow whilst been thermally stratified in such a way that the warmer (hence lighter) fluid is located at the top and the colder (hence denser) fluid at the bottom.

Linearizing \eqref{eq:3DBoussinesq} near $(U^s,\Theta^s)$, hence
writing $V=U^s+u$, $\Theta=\Theta^s-\sqrt{\alpha/\mathfrak{g}}\,\theta$, and neglecting nonlinear contributions, we find the system
\begin{equation}\label{eq:3DBoussinesqLin}
\begin{cases}
\de_t u+y\de_xu+u^{2}\hat x+2\nabla (-\Delta)^{-1}\de_x u^{2}+\beta\nabla(-\Delta)^{-1}\de_y\theta=\nu\Delta u-\beta\theta \hat y, \\
\de_t\theta+y\de_x\theta - \beta u^{2}=\kappa\Delta\theta,
\end{cases}
\end{equation}
where $\beta=\sqrt{\alpha\mathfrak{g}}$ is the \textit{Brunt-V\"ais\"al\"a} frequency, which describes the frequency at which a small parcel of fluid, when displaced vertically in a stable, stratified environment, oscillates back and forth due to buoyancy forces. It is also worth noticing that the terms $2\nabla (-\Delta)^{-1}\de_x u^{2}$ and $\beta\nabla(-\Delta)^{-1}\de_y\theta$
appear as the linearization of the pressure term.

\subsubsection{The lift-up effect}\label{subsubsec:theliftup}
In their groundbreaking work \cite{EP75}, Ellingsen and Palm unveiled a fundamental linear mechanism capable of inducing instability in shear flows, a phenomenon now recognized as the \emph{lift-up effect}. Kline et al. \cite{KRSR67} coined this term while studying turbulent boundary layers, observing that in the laminar sublayer, \emph{streaks} gradually moved away from the slow-moving fluid near the bottom wall to the faster fluid above. They proposed that vortex filaments, generated perpendicular to the streaks due to the flow's 3D nature, drove this upward motion. Notably, the lift-up effect is absent in 2D flows due to incompressibility. For further details, refer to \cites{EP75,L80,KRSR67,B14,P&al20} and related references. Building on Brandt's review \cite{B14}, we quantify this mechanism by defining streak solutions and elucidating their growth origins.

For any function $\varphi:\TT\times\RR\times \TT\to\RR$, we define
\begin{equation}\label{eq:defzeromode}
\varphi_{0}=\frac{1}{2\pi}\int_{\TT}\varphi(x,\cdot,\cdot)\dd x, \qquad \varphi_{\neq}=\varphi-\varphi_0.
\end{equation}
Solutions to \eqref{eq:3DBoussinesqLin} that are independent of $x$ are called \emph{streaks}. These correspond to $x$-averaged solutions, which we will also refer to as the zero modes.

In view of periodicity, streaks solutions to \eqref{eq:3DBoussinesqLin} satisfy
\begin{equation}\label{eq:k=0mode}
\begin{cases}
\de_t u^{1}_0+u^{2}_0=\nu\Delta_{y,z} u^{1}_0,\\
\de_t u^{2}_0+\beta\de_y(-\Delta_{y,z})^{-1}\de_y\theta_0=\nu\Delta_{y,z} u^{2}_0-\beta\theta_0, \\
\de_t u^{3}_0+\beta\de_z(-\Delta_{y,z})^{-1}\de_y\theta_0=\nu\Delta_{y,z} u^{3}_0, \\
\de_t\theta_0+\beta u^{2}_0=\kappa\Delta_{y,z}\theta_0.
\end{cases}
\end{equation}
When $\beta=0$, the above system reduces to the homogeneous Navier-Stokes equations. In this case, $(u^{2}_0,u^{3}_0)$ satisfy the two-dimensional heat the equation, while $u^{1}_0$ is forced by $u^{2}_0$ in a linear way. The explicit solution can be written as
\begin{equation}\label{eq:liftupsol}
u_0(t)=\e^{\nu\Delta_{y,z}t}\left(u^{1}_0(0)-tu_0^{2}(0),u_0^{2}(0),u_0^{3}(0)\right)
\end{equation}
This causes a large (as $\nu\to 0$) transient growth (the lift-up effect) of order $t$, at least for times shorter than $O(\nu^{-1})$.  This is in fact the main instability mechanism in 3D, see \cites{BGM20,BGM22}.

When $\beta>0$, the coupling between $u^{2}_0$ and $\theta_0$ creates an oscillatory behavior that suppresses the transient growth. In particular, all the components of the system decay as the heat equation prescribes (see Theorem \ref{thm:noliftup} below). This is similar to what happens for example in magnetohydrodynamics, where the presence of the magnetic field interacts with the velocity field in such a way that that the lift-up effect is cancelled (see \cite{L18} for more details).

\subsection{Main results}
In this paper we present and discuss two results regarding the solutions of \eqref{eq:3DBoussinesqLin}. It is well known that, due to the effect of mixing of the background Couette flow, the nonzero modes of the solution of the 3D Navier-Stokes equations experience dissipation at times $O(\nu^{-1/3})$, which are much shorter than the purely dissipative time-scale $O(\nu^{-1})$. This is true both in 2D \cites{BMV16,BVW18,MZ22,MZ20,CLWZ20} and in 3D \cites{BGM22,WZ18,BGM20,BGM15}. We show here that stratification in \eqref{eq:3DBoussinesqLin} does not spoil this mechanism.
\begin{theorem}[Linear enhanced dissipation]\label{thm:enhancelin}
Let $\beta>1/2$, and define
\begin{equation}\label{eq:Cbeta}
C_\beta:= \left[\frac{2\beta+1}{2\beta-1}\exp\left(  \frac{1}{  2\beta-1}\right)\right]^{1/2}.
\end{equation}
Assume further that $\nu,\kappa>0$ satisfy
\begin{equation}\label{eq:kappanu}
    \frac{\max\{\nu,\kappa\}}{\min\{\nu,\kappa\}}<4\beta-1,
\end{equation}
and define the strictly positive number
\begin{equation}\label{eq:lambdarate}
\lambda_{\nu, \kappa}:=\min\{\nu,\kappa\}\left(1-\frac{1}{4\beta}-\frac{1}{4\beta}\frac{\max\{\nu,\kappa\}}{\min\{\nu,\kappa\}}\right).
\end{equation}
Then
\begin{equation}\label{eq:Linenhanced}
\|(u^{1},u^{3})_{\neq}(t)\|_{L^2} +\l t \r^{3/2}\|u^{2}_{\neq}(t)\|_{L^2} +\l t \r^{1/2}\|\theta_{\neq}(t)\|_{L^2}\lesssim \e^{- \frac{1}{24} \lambda_{\nu,\kappa}t^3 }\left[  \| u_{\neq}(0)\|_{H^3} +\| \theta_{\neq}(0)\|_{H^{3}}\right], 
\end{equation}
for every $t\geq 0$.
\end{theorem}

\begin{remark}
The condition $\beta>1/2$ is in fact related to the Miles-Howard criterion \cites{M61,H61} of spectral stability for stratified shear flows. This requires the Richardson number (which equals $\beta^2$ in the particular steady state \eqref{eq:steadystate}) to be greater than $1/4$. When dissipation is present, our method of proof, based on a symmetrization technique introduced in \cite{BCZD20}, requires the stronger \eqref{eq:kappanu}, at least when $\nu\neq \kappa$.
\end{remark}

\begin{remark}
The enhanced dissipation rate \eqref{eq:lambdarate} gives a time-scale $O(\min\{\nu,\kappa\}^{-1/3})$, which is consistent with the expected one mentioned above in the homogeneous case. Since the constraint \eqref{eq:kappanu} and estimate \eqref{eq:Linenhanced} are stable in the limit $\nu=\kappa\to 0$, Theorem \ref{thm:enhancelin} provides an inviscid damping estimate with algebraic rates. These have similarities with the 2D rates derived in \cite{BCZD20}, in that they are slower than those for the Couette flow in the 2D homogeneous Euler equations \cite{BM15}.
\end{remark}

We shift the focus now to the analysis of solutions to \eqref{eq:k=0mode}. Our second theorem covers the case of streak solutions and encapsulates the main result of this paper, namely the suppression of lift-up effect. 
\begin{theorem}[Suppression of lift-up]\label{thm:noliftup}
   Let $\beta>0$ and assume $\nu,\kappa >0$. Then, the solution $(u_0(t),\theta_0(t))$ of \eqref{eq:k=0mode} with initial datum $(u_0(0),\theta_0(0))\in H^4$ satisfies 
\begin{equation}\label{eq:noliftup}
       \norm{(u,\theta)_{0}(t)}_{L^2} \lesssim_\beta \e^{-\min\{\nu,\kappa\}t} \norm{(u,\theta)_{0}(0)}_{H^4},
   \end{equation}
   for every $t\geq0$.
   In particular, the dependence of the constant $\beta$ appears only in the bound of the first component $u^1_0$ and reads
   \begin{equation}\label{eq:deponbeta}
       \|u^1_0(t)\|_{L^2}\lesssim \e^{-\min\{\nu,\kappa\}t}\left[\|u^1_0(0)\|_{H^4}+\frac{1}{\beta}(\|u^2_0(0)\|_{H^4}+\|\theta_0(0)\|_{H^4})\right].
   \end{equation}
\end{theorem}
\begin{remark}
    In the homogeneous inviscid case, the explicit solution \eqref{eq:liftupsol} shows that  $u^1_0$ can grow linearly in time. 
    On the contrary, \eqref{eq:deponbeta} entails a uniform bound in time, as long as the velocity interacts with the temperature, namely when $\beta>0$. Moreover, the exponential decay $\e^{-\min\{\nu,\kappa\}t}$ is consistent with the purely dissipative behaviour of streaks in the homogeneous Navier-Stokes equations.
\end{remark}

Theorems \ref{thm:enhancelin} and \ref{thm:noliftup} form the essential components of the linear theory, paving the way for our upcoming work \cite{CZDZnonlin23}. This forthcoming result establishes the existence of a nonlinear stability threshold in Sobolev regularity, implying that solutions to \eqref{eq:3DBoussinesq}, which are close to the stratified Couette flow \eqref{eq:steadystate}, undergo enhanced dissipation on time-scales $O(\min\{\nu,\kappa\}^{-1/3})$. For analogous results in the homogeneous Navier-Stokes equations, we refer the reader to \cites{BMV16,BVW18,MZ22,MZ20,CLWZ20,BGM22,WZ18,BGM20,BGM15}.

\subsection{Outline of the paper}
The paper is now structured as follows: Section \ref{sec:nonzero} is devoted to the proof of Theorem \ref{thm:enhancelin}. At the beginning of the section we introduce tools and notations which will be used throughout the proof and the paper itself. Afterwards, we discuss the proof, which is based on energy methods. More precisely, we introduce an auxiliary set of variables that symmetrizes the coupled system for $u^{2}_{\neq}$ and $\theta_{\neq}$. After that we define an appropriate energy functional and we deduce enhanced dissipation for these two variables. Finally, this translates to enhanced dissipation for the remaining two variables $u^{1}_{\neq}$ and $u^{3}_{\neq}$.
Section \ref{sec:zerolin} contains the proof of Theorem \ref{thm:noliftup}, which relies on the explicit solution of the zero mode system, and  uses explicitly the oscillating behaviour expected from the coupling between $u^{2}_0$ and $\theta_0$. This is enough to show that there is no time growth before the dissipation time-scale is $O(\min\{\nu,\kappa\}^{-1})$.

\section{Nonzero modes analysis}\label{sec:nonzero}
The transport structure of \eqref{eq:3DBoussinesqLin} suggests the linear change of variable
\begin{align}\label{eq:changelin}
X=x-yt,\qquad Y=y, \qquad Z=z.
\end{align}
We will use the convention of using capital letters when considering any function $f$ in the moving frame \eqref{eq:changelin}, 
hence defining $F$ by $F(t,X,Y,Z)=f(t,x,y,z)$. The corresponding differential operators change accordingly as
\begin{equation}
\de_x=\de_X,\qquad \de_y=\de_Y-t\de_X,\qquad \de_z=\de_Z,
\end{equation}
and we denote
\begin{equation}
\Delta_L:= \de^2_X+\left(\de_Y-t\de_X\right)^2+\de_Z^2
\end{equation}
the Laplacian in the new framework.
\subsection{Notation and conventions}
We introduce the following notation for the Fourier transform of a function $\varphi=\varphi(x,y,z)$. Given $(k,\eta,l)\in \ZZ\times\RR\times\ZZ$, define 
\begin{equation}
     \varphi_{k,l}(\eta):=\frac{1}{4\pi^2}\int_{\TT\times\RR\times\TT}\e^{-i(kx+\eta y+lz)}\varphi(x,y,z)\dd x \dd y \dd z.
\end{equation}
The function $\varphi$ in real variable can then be recovered  via 
\begin{equation}
\varphi(x,y,z)=\sum_{(k,l)\in\ZZ^2}\int_{\RR}\e^{i(kx+\eta y+lz)} \varphi_{k,l}(\eta)\dd \eta.
\end{equation}
With a slight abuse of notation, we denote the $k=0$ mode with only one index (that is, $\varphi_0$ as in \eqref{eq:defzeromode}), without distinguishing between the original function and its Fourier transform.

The $L^2$ inner product and norm are denoted by $\l\cdot,\cdot\r$ and $\|\cdot\|$, respectively.
A subscript will be added to the norm when referring to the norm  in $y$ only,  i.e. $L^2_y$.
Moreover, we define 
\begin{equation}
    |k,\eta,l|^2:=k^2+\eta^2+l^2, \quad \l (k,\eta,l)\r:=\sqrt{1+|k,\eta,l|^2}.
\end{equation}
Following this, we denote the $H^s$ norm, for $s>0$, as
\begin{equation}
    \|\cdot\|_{H^s}:=\|\l(k,\eta,l)\r^s\cdot\|.
\end{equation}
We define the Fourier symbol of $-\Delta_L$ as 
\begin{equation}
    p(t,k,\eta,l)=k^2+(\eta-kt)^2+l^2.
\end{equation}
Note that $p$ is time dependent and its time derivative is
\begin{equation}
    p'(t,k,\eta,l)=-2k(\eta-kt).
\end{equation}
Finally, we use the notation $a\lesssim b$ if there exists a positive constant $C$ such that $a\leq C b$.

\subsection{Symmetric variables}
We begin the analysis of the linearized system \eqref{eq:3DBoussinesqLin} from the \emph{nonzero $x$-modes}, corresponding
to $k\neq0$ on the Fourier side.
The most strongly coupled components are $u^{2}$ and $\theta$. In fact, it is convenient to work
with the new variable $\q=\Delta u^{2}$. This new unknown, already noticed by Kelvin \cite{K87} allows us to exploit a useful cancellation between the transport and the pressure term.
The system for $(\q,\theta)$ reads
\begin{equation}\label{eq:linBoussqtheta}
\begin{cases}
    \de_t\q+y\de_x \q+\beta\Delta_{x,z}\theta=\nu\Delta \q,\\
    \de_t\theta+y\de_x\theta + \beta(-\Delta)^{-1} \q=\kappa\Delta\theta.
\end{cases}
\end{equation}
In the moving frame \eqref{eq:changelin}, this can be written on the Fourier side as
\begin{equation}\label{eq:linBoussqthetaMov}
\begin{cases}
    \de_t\Q -\beta|k,l|^2\Theta=-\nu p \Q, \\
    \de_t\Theta +\beta p^{-1} \Q=-\kappa p\Theta.
\end{cases}
\end{equation}
Is it worth writing here how the equations for $u^1$ and $u^3$ translate in this setting.
\begin{equation}\label{eq:u1u3moving}
    \begin{cases}
    \de_t U^{1}+\nu p U^{1} =p^{-1}\Q+2k^2p^{-2}\Q+\beta k (\eta-kt) p^{-1}\Theta,\\
    \de_t U^{3}+\nu p U^{3}= 2kl p^{-2} \Q+\beta l(\eta-kt)p^{-1}\Theta.
    \end{cases}
\end{equation}
For $\nu=\kappa=0$, the structure of this system is precisely the one studied in the 2D Boussinesq equations \cites{BCZD20, BBCZD21}. Following
this approach, we symmetrize \eqref{eq:linBoussqthetaMov} by introducing the new unknowns
\begin{equation}\label{eq:SimVar}
 \GG:=\frac{1}{|k,l|^{1/2}p^{1/4}}\Q, \qquad  \Gamma:=|k,l|^{1/2}p^{1/4}\Theta,
\end{equation} 
in order to obtain
\begin{align}
\de_t \GG +\frac{1}{4}\frac{p'}{p} \GG-\beta |k,l| p^{-\frac{1}{2}} \Gamma&=-\nu p  \GG,\label{eq:simQ}\\
\de_t \Gamma-\frac{1}{4}\frac{p'}{p} \Gamma+\beta |k,l|p^{-\frac{1}{2}} \GG &=-\kappa p \Gamma\label{eq:simT}.
\end{align}

\subsection{The main energy functional}
The analysis of \eqref{eq:simQ}-\eqref{eq:simT} relies on the introduction of the energy functional
\begin{align}\label{def:pointwise-functional-Couette}
\sfE(t)=\frac12\left[|\GG(t)|^2+|\Gamma(t)|^2+\frac{1}{2\beta}\frac{p'} {|k,l|p^{\frac12}}  \Re  \left(\GG(t) \overline{\Gamma(t)}\right)\right].
\end{align}
Since $|p'/(|k,l|p^{1/2})|\leq 2$, the energy   is coercive for $\beta>1/2$ with
\begin{align}\label{eq:coercive-pointwise}
\frac12\left(1-\frac{1}{2\beta}\right)\left[|\GG|^2+|\Gamma|^2\right] \leq \sfE \leq\frac12\left(1+\frac{1}{2\beta}\right)\left[|\GG|^2+|\Gamma|^2\right],
\end{align}
and   satisfies the  identity
\begin{align}\label{eq:energqueq}
\ddt \sfE=\frac{1}{4\beta}\de_t\left(\frac{p'} {|k,l| p^{\frac12}}\right)\Re\left(\GG\overline{\Gamma}\right) - \nu p |\GG|^2 - \kappa p|\Gamma|^2
-\frac{\nu+\kappa}{4\beta}\frac{p'} {|k|p^{\frac12}}  p  \Re  \left(\GG \overline{\Gamma}\right) .
\end{align}
From this, we can infer the key information needed to prove Theorem \ref{thm:enhancelin}.

\begin{proposition}\label{prop:enhancedQT}
Let $\beta>1/2$.
If $\nu,\kappa>0$ comply with \eqref{eq:kappanu}, then  there holds 
\begin{equation} \label{eq:energenhanced}
|\GG(t)|^2+|\Gamma(t)|^2\leq C^2_\beta\e^{- \frac{1}{12 } \lambda_{\nu,\kappa}  k^2t^3 } \left[|\GG(0)|^2+|\Gamma(0)|^2\right] ,
\end{equation}
for every $t\geq 0$, with $C_\beta$ and $\lambda_{\nu,\kappa}$ given respectively in \eqref{eq:Cbeta} and \eqref{eq:lambdarate}.
\end{proposition}

\begin{proof}
From the coercivity bound \eqref{eq:coercive-pointwise}, identity \eqref{eq:energqueq} implies that 
\begin{equation} \label{eq:energqueq1}
\ddt \sfE\leq \frac{1}{2(2\beta-1)}\de_t\left(\frac{p'} {|k,l| p^{\frac12}}\right) \sfE  - \nu p |\GG|^2 - \kappa p|\Gamma|^2
-\frac{\nu+\kappa}{4\beta}\frac{p'} {|k,l| p^{\frac12}}  p  \Re  \left(\GG \overline{\Gamma}\right).
\end{equation}
Moreover, from $|p'/(|k,l| p^{1/2})|\leq 2$ we find that
\begin{align}
\frac{\nu+\kappa}{4\beta}\frac{p'} {|k,l| p^{\frac12}}  p  \Re  \left(\GG \overline{\Gamma}\right) \leq \frac{\nu+\kappa}{4\beta} p\left[|\GG|^2+|\Gamma|^2\right].
\end{align}
Thanks to \eqref{eq:kappanu}, the constants 
\begin{equation}
\lambda_\nu:= \nu -  \frac{\nu+\kappa}{4\beta} , \qquad \lambda_\kappa:= \kappa -  \frac{\nu+\kappa}{4\beta}, \qquad \lambda_{\nu, \kappa} =\min \{\lambda_\nu,\lambda_\kappa\}
\end{equation}
are all strictly positive, 
and \eqref{eq:coercive-pointwise} and \eqref{eq:energqueq1} imply
\begin{align} 
\ddt \sfE
\leq \frac{1}{2(2\beta-1)}\de_t\left(\frac{p'} {|k,l| p^{\frac12}}\right) \sfE- \frac{4\beta}{2\beta +1 }\lambda_{\nu,\kappa} p \sf\,E.
\end{align}
Using that $|p'/(|k,l|p^{1/2})|\leq 2$, the first term on the right-hand side is integrable, therefore
\begin{equation}\label{eq:energyboundP}
    \sfE(t)\leq \exp\left(\frac{1}{ 2\beta-1 } \right)\exp\left(- \frac{4\beta}{2\beta+1}\lambda_{\nu,\kappa} \int_0^t p(s)\dd s \right)\sfE(0) 
\end{equation}
Since
\begin{equation}\label{eq:lbintegrp}
    \int_0^t p(s)\dd s = (k^2+l^2)t + t\left((\eta-\frac12kt)^2+\frac{1}{12}k^2t^2\right) \geq \frac{1}{12}k^2t^3,
\end{equation}
we obtain
\begin{align} \label{eq:energybound}
\sfE(t)\leq \exp\left(\frac{1}{ 2\beta-1 } \right)\exp\left(- \frac{\beta}{3(2\beta +1) }\lambda_{\nu,\kappa} k^2t^3 \right)\sfE(0) ,
\end{align}
and \eqref{eq:energenhanced} follows.
\end{proof}
Now, to derive estimates on $u^2$ and $\theta$ in the original variables, as in  Theorem \ref{thm:enhancelin}, we argue as follows.
Recalling $Q=\Delta_L U^{2}$, the symmetric change of variables \eqref{eq:SimVar} and using the bound \eqref{eq:energybound} on the energy functional $\sfE$, we compute 
\begin{align}
\|u^{2}_{k,l}(t)\|_{L_y^2}^2
&=\| \Delta_L^{-1}\Q_{k,l}(t)\|_{L_y^2}^2
=\int_\mathbb{R} |k,l| p^{-3/2}|\GG_{k,l}(t)|^2 \dd \eta\\
&\lesssim  \e^{- \frac{1}{12 } \lambda_{\nu,\kappa}  t^3 } |k,l| \int_\mathbb{R} p^{-{3/2}}\left[|\GG_{k,l}(0)|^2+|\Gamma_{k,l}(0)|^2\right]  \dd \eta\\
&\lesssim  \e^{- \frac{1}{12 } \lambda_{\nu,\kappa}  t^3 }  \int_\mathbb{R} p^{-{3/2}}\left[|k,\eta,l|^{3}|u^{2}_{k,l}(0)|^2+|k,l|^2 |k,\eta,l||\theta_{k,l}(0)|^2\right]  \dd \eta\\
&\lesssim  \frac{\e^{- \frac{1}{12 } \lambda_{\nu,\kappa}t^3 } }{\l t\r^3} \left[\| u^{2}(0)\|^2_{H^{3}} +\| \theta(0)\|^2_{H^{3}} \right].
\end{align}
Here we used that $p^{-1}\leq \l t\r^{-2}|k,\eta,l|^2$.
Similarly,
\begin{align}
\|\theta_{k,l}(t)\|_{L_y^2}^2
&=\|\Theta_{k,l}(t)\|_{L_y^2}^2
=\int_\mathbb{R} \frac{1}{\beta|k,l| p^{1/2}}|\Gamma_{k,l}(t)|^2 \dd \eta\\
&\lesssim \e^{- \frac{1}{12 } \lambda_{\nu,\kappa}  t^3 }  \frac{1}{|k,l|}\int_\mathbb{R} p^{-{1/2}}\left[|\GG_{k,l}(0)|^2+|\Gamma_{k,l}(0)|^2\right]  \dd \eta\\
&\lesssim \e^{- \frac{1}{12 } \lambda_{\nu,\kappa}   t^3 } \int_\mathbb{R} p^{-{1/2}}\left[ \frac{|k,\eta,l|^{3}}{ |k,l|^2}|u^{2}_{k,l}(0)|^2+  |k,\eta,l||\theta_{k,l}(0)|^2\right]   \dd \eta\\
&\lesssim  \frac{\e^{- \frac{1}{12 } \lambda_{\nu,\kappa}t^3 } }{\l t\r} \left[\| u^{2}(0)\|^2_{H^2} +\| \theta(0)\|^2_{H^1} \right].
\end{align}
Therefore, estimate \eqref{eq:Linenhanced} for $u^2$ and $\theta$ follows from a summation in $(k,l)$, with $k\neq 0$.

\subsection{Enhanced dissipation for $u^{1}$ and $u^{3}$}
Starting from \eqref{eq:u1u3moving}, we rewrite the equation in terms of the symmetric variables \eqref{eq:SimVar} as 
\begin{align}
\de_t U^{1}+\nu p U^{1} &=|k,l|^{1/2}p^{-{3/4}}\GG+2k^2   |k,l|^{1/2}p^{-7/4}\GG+\beta\frac{k (\eta-kt) p^{-5/4}}{|k,l|^{1/2}}\Gamma,\label{eq:U1}\\
\de_t U^{3}+\nu p U^{3}&= 2kl p^{-7/4} |k,l|^{1/2}\GG+\beta\frac{ l(\eta-kt)p^{-5/4}}{|k,l|^{1/2}}\Gamma.\label{eq:U3}
\end{align}
We can then prove the following enhanced dissipation result, under the same assumptions of Proposition \ref{prop:enhancedQT}.
\begin{proposition}\label{prop:enhancedu1u3}
Let $\beta>1/2$.
If $\nu,\kappa>0$ comply with \eqref{eq:kappanu}, then for $j\in\{1,3\}$  there holds 
\begin{align} \label{eq:energenhancedu1u3}
 |U^{j}(t)| \leq \e^{-\frac{1}{24}\lambda_{\nu,\kappa} k^2t^3 } \left[ |U^{j}(0)| +6 C_\beta(3+\beta)\left(\frac{k^2+l^2}{k^6}\right)^{1/4}  \left[|\GG(0)|^2+|\Gamma(0)|^2\right]^\frac12\right], 
\end{align}
for every $t\geq 0$, with $C_\beta$ and $\lambda_{\nu,\kappa}$ given respectively in \eqref{eq:Cbeta} and \eqref{eq:lambdarate}.
\end{proposition}

\begin{proof}
A standard energy estimate on \eqref{eq:U1}-\eqref{eq:U3} gives
\begin{align}
\ddt |U^{1}|+\nu p  |U^{1}| \leq 3|k,l|^{1/2}p^{-{3/4}}|\GG|  +\beta |k,l|^{1/2}    p^{-3/4}|\Gamma| ,\\
\ddt |U^{3}|+\nu p  |U^{3}|\leq  2|k,l|^{1/2} p^{-3/4} |\GG| +\beta |k,l|^{1/2}    p^{-3/4} |\Gamma|.
\end{align}
It is enough to discuss the inequality for $|U^1|$ since the one for $|U^3|$ is essentially the same.
Following the proof of Proposition \ref{prop:enhancedQT} and using \eqref{eq:energyboundP} lead to 
\begin{equation}
    |\GG(t)|^2+|\Gamma(t)|^2\leq C^2_\beta\e^{- \frac{4\beta}{2\beta+1} \lambda_{\nu,\kappa}  \int_0^tp(s)\dd s } \left[|\GG(0)|^2+|\Gamma(0)|^2\right],
\end{equation}
so that we end up with
\begin{align}
\ddt |U^{1}|+\nu p  |U^{1}| &\leq C_\beta(3+\beta)|k,l|^{1/2}p^{-{3/4}}  \e^{- \frac{2\beta}{2\beta+1 } \lambda_{\nu,\kappa}  \int_0^tp(s)\dd s }\left[|\GG(0)|^2+|\Gamma(0)|^2\right]^\frac12.
\end{align}
Via the Gronwall inequality, we deduce that
\begin{align*}
  |U^1|&\leq \e^{-\nu\int_0^t p(s)\dd s}\left[|U^1(0)|+\bar{C}_\beta\int_0^tp^{-3/4}(s)\e^{\left(\nu - \frac{2\beta}{2\beta+1 } \lambda_{\nu,\kappa}\right)\int_0^s p(\tau)\dd \tau}\dd s\left[|\GG(0)|^2+|\Gamma(0)|^2\right]^\frac12\right],
\end{align*}
where $\bar C _\beta =C_\beta(3+\beta)|k,l|^{1/2}$.
Note that the function $s\mapsto\int_0^sp(\tau)\dd\tau$ is increasing  and $\nu - \frac{2\beta}{2\beta+1 } \lambda_{\nu,\kappa}$ is always positive.
Therefore, the monotonicity of the exponential inside the integral, estimate \eqref{eq:lbintegrp}, and
\begin{equation}
\int_{\RR}p^{-3/4}\dd t=\int_{\RR} \frac{1}{(k^2+(\eta-kt)^2+l^2)^{3/4}}\dd t \leq 6 |k|^{-3/2},
\end{equation}
imply 
\begin{align}
 |U^{1}(t)| &\leq  \e^{-\frac{1}{24}\lambda_{\nu,\kappa} k^2t^3 } \left[ |U^{1}(0)| +6 C_\beta(3+\beta)\left(\frac{|k,l|^2}{k^6}\right)^{1/4}  \left[|\GG(0)|^2+|\Gamma(0)|^2\right]^\frac12\right].
\end{align}
This concludes the proof.
\end{proof}

To finish the proof of \eqref{eq:Linenhanced}, we again have to translate \eqref{eq:energenhancedu1u3} to the original variables. We do this only for $u^{1}$, as the estimate for $u^{3}$ is exactly
the same. Recalling \eqref{eq:SimVar}, we have
\begin{align}
\|u^{1}_{k,l}(t)\|_{L_y^2}^2
&\lesssim \e^{- \frac{1}{12 } \lambda_{\nu,\kappa}  t^3 }  \int_\mathbb{R}    |U^{1}_{k,l}(0)|^2 +\left(\frac{k^2+l^2}{k^6}\right)^{1/2}  \left[|\GG_{k,l}(0)|^2+|\Gamma_{k,l}(0)|^2\right]  \dd \eta\\
&\lesssim \e^{- \frac{1}{12 } \lambda_{\nu,\kappa}   t^3 } \int_\mathbb{R}\left[|u^{1}_{k,l}(0)|^2+   \frac{ |k,\eta,l|^3}{|k,l|  }  |u^{2}_{k,l}(0)|^2+|k,l|   |k,\eta,l||\theta_{k,l}(0)|^2\right]   \dd \eta\\
&\lesssim  \e^{- \frac{1}{12 } \lambda_{\nu,\kappa}t^3 }  \left[\| u^{1}(0)\|^2_{L^2} + \| u^{2}(0)\|^2_{H^{3/2}} +\| \theta(0)\|^2_{H^1} \right],
\end{align}
which is precisely included in \eqref{eq:Linenhanced}, after a summation over $(k,l)$, with $k\neq 0$. 

\section{Zero modes analysis}\label{sec:zerolin}
We now switch our attention to the study of the streak solutions, which satisfy \eqref{eq:k=0mode}. This analysis reveals 
one of the crucial differences compared to the homogeneous Navier-Stokes equations,
namely the absence of the lift-up effect. 
On the Fourier side, \eqref{eq:k=0mode} reads
\begin{equation}\label{eq:k=0modeFou}
\begin{cases}
\displaystyle\de_t u^{1}_0=-u^{2}_0-\nu |\eta,l|^2u^{1}_0,\\
\displaystyle\de_t u^{2}_0=-\beta\frac{l^2}{|\eta,l|^2}\theta_0-\nu |\eta,l|^2u^{2}_0, \\
\displaystyle\de_t u^{3}_0=\beta \frac{\eta l}{|\eta,l|^2} \theta_0-\nu |\eta,l|^2 u^{3}_0, \\
\displaystyle\de_t\theta_0=\beta u^{2}_0-\kappa |\eta,l|^2\theta_0.
\end{cases}
\end{equation}
The special case $l=0$ is the simplest: $u^{2}_{0,0}=0$ due to incompressibility, and the above system reduces to three decoupled
heat equations in $\RR$. When $l\neq0$, the main point is to understand the dynamics of the $(u^{2}_0,\theta_0)$ system, via the analogue of the symmetric variables
\eqref{eq:SimVar}, which in this case become
\begin{equation}\label{eq:SimVarZero}
 g_0:=\frac{|\eta,l|^{3/2}}{|l|^{1/2}}u^{2}_0, \qquad  \gamma_0:=|\eta,l|^{1/2}|l|^{1/2}\theta_0.
\end{equation} 
It is in fact convenient to also rescale the other variables as
\begin{equation}\label{eq:SimVarZero2}
 f_0:=-\frac{|\eta,l|^{3/2}}{|l|^{1/2}}u^{1}_0, \qquad  h_0:=\frac{|\eta,l|^{5/2}|l|^{1/2}}{\beta l }u^{3}_0.
\end{equation} 
In these new variables, system \eqref{eq:k=0modeFou} can now be written in the compact form
\begin{equation}
\ddt \begin{pmatrix}
f_0	\\
h_0 \\
g_0\\
\gamma_0
\end{pmatrix}
=
\begin{pmatrix}
N & S \\
0 & M
\end{pmatrix}
\begin{pmatrix}
f_0	\\
h_0 \\
g_0\\
\gamma_0
\end{pmatrix},
\end{equation} 
where  
\begin{equation}
N=-\nu|\eta,l|^2
\begin{pmatrix}
1 & 0\\
0 &1
\end{pmatrix},
\qquad
S=
\begin{pmatrix}
1 & 0\\
0 & \eta
\end{pmatrix},
\qquad
M=
\begin{pmatrix}
-\nu|\eta,l|^2 & -\beta \frac{|l|}{|\eta,l|}\\
\beta \frac{|l|}{|\eta,l|} & -\kappa|\eta,l|^2
\end{pmatrix}.
\end{equation}
The solution to this ODE can be explicitly computed as
\begin{equation}\label{eq:zeromodesolution}
\begin{pmatrix}
f_0(t)	\\
h_0(t) \\
g_0(t)\\
\gamma_0(t)
\end{pmatrix}
=
\exp\left[
\begin{pmatrix}
N & S \\
0 & M
\end{pmatrix}
t\right]
\begin{pmatrix}
f_0(0)	\\
h_0(0) \\
g_0(0)\\
\gamma_0(0)
\end{pmatrix}.
\end{equation}
Since $N$ and $M$ commute,  the exponential matrix has the   form
\begin{equation}\label{eq:bigexpt}
\exp\left[
\begin{pmatrix}
N & S \\
0 & M
\end{pmatrix}
t\right]
=
\begin{pmatrix}
\e^{Nt} & S (N-M)^{-1}\left(\e^{Nt}-\e^{Mt}\right)\\
0 & \e^{Mt}
\end{pmatrix}.
\end{equation}
The matrix $N-M$ is indeed invertible, since 
\begin{equation}
\det (N-M)
=
\beta^2 \frac{|l|^2}{|\eta,l|^2}\neq 0.
\end{equation}
The exponential matrix $\e^{Mt}$ has a quite complicate expression in terms of hyperbolic functions. Defining
\begin{equation}\label{eq:abc}
\sfa:=\frac{|\nu-\kappa||\eta,l|^2}{2}, \qquad \sfb:=\frac{\beta |l|}{|\eta,l|}, \qquad \sfc:= \sqrt{\sfa^2-\sfb^2},
\end{equation}
we can explicitly write 
\begin{equation}\label{eq:expMt}
  \e^{Mt}=
   \e^{-\frac{\nu+\kappa}{2} |\eta,l|^2 t} 
   \begin{pmatrix}
\cosh(\sfc t)- \frac{\nu-\kappa}{2\sfc}|\eta,l|^2\sinh(\sfc t)	& 	- \frac{\sfb}{\sfc}\sinh(\sfc t)\\
  \frac{\sfb}{\sfc}\sinh(\sfc t)&	\cosh(\sfc t)+ \frac{\nu-\kappa}{2\sfc}|\eta,l|^2\sinh(\sfc t)
    \end{pmatrix}.
\end{equation}
To better visualize the oscillating nature of the solution the reader can focus on the special case $\nu=\kappa$. With this restriction we have $\sfa=0$ and $\sfc=i\sfb$, so that 
the oscillatory nature of the terms in \eqref{eq:expMt} become apparent.
Continuing with the most general setting, we have
\begin{equation}\label{eq:solutionZero}
(N-M)^{-1}\left(\e^{Nt}-\e^{Mt}\right) =  \begin{pmatrix}
    \sfm_{11} &\sfm_{12} \\
    -\sfm_{12} & \sfm_{22} 
\end{pmatrix},
\end{equation}
where
\begin{align}
    \sfm_{11}&= \frac{\nu -\kappa}{\sfb^2}|\eta,l|^2 \left( \e^{-\frac{\nu+\kappa}{2}|\eta,l|^2 t} \phi_-(\sfc t)+\e^{-\nu |\eta,l|^2t }\right)
    +\frac{\e^{-\frac{\nu+\kappa}{2}|\eta,l|^2 t}}{\sfc} \sinh(\sfc t),\\
    \sfm_{12}&=\frac{1}{\sfb}\left(\e^{-\frac{\nu+\kappa}{2}|\eta,l|^2 t}\phi_-(\sfc t)-\e^{-\nu|\eta,l|^2 t} \right),\\
    \sfm_{22}&=\frac{1}{\sfc} \e^{-\frac{\nu+\kappa}{2} |\eta,l|^2t}\sinh(\sfc t),
\end{align}
and
\begin{align}
\phi_\pm(\sfc t)=\cosh(\sfc t)\pm \frac{\nu-\kappa}{2\sfc}|\eta,l|^2\sinh(\sfc t).
\end{align}
With this notation we have 
\begin{equation}
  \e^{Mt}=
   \e^{-\frac{\nu+\kappa}{2} |\eta,l|^2 t} 
   \begin{pmatrix}
\phi_-(\sfc t)	& 	- \frac{\sfb}{\sfc}\sinh(\sfc t)\\
  \frac{\sfb}{\sfc}\sinh(\sfc t)&	\phi_+(\sfc t)
    \end{pmatrix}.
\end{equation}
We now aim to show that the oscillations present in the terms above do prevent the lift-up effect and hence yield bounds that are independent of $\nu,\kappa$. More specifically, we need to establish uniform-in-time bounds on each entry of the matrix.

\subsection{Proof of Theorem \ref{thm:noliftup}}
To show uniform-in-time bounds on the entries of matrix \eqref{eq:bigexpt}, we start by noticing that 
\begin{equation}
    \nu +\kappa =2\min\{\nu,\kappa\}+|\nu-\kappa|.
\end{equation}
This allows us to extract the dissipation factor from the matrix \eqref{eq:bigexpt} and to recover the heat decay expected for times larger than $\min\{\nu,\kappa\}^{-1}$.
After collecting $\e^{-\min\{\nu,\kappa\}|\eta,l|^2t}$, the matrix looks like 
\begin{equation}
    \begin{pmatrix}
       \e^{-(\nu-\min\{\nu,\kappa\})|\eta,l|^2t} & 0 & \e^{\min\{\nu,\kappa\}|\eta,l|^2t}\sfm_{11} & \e^{\min\{\nu,\kappa\}|\eta,l|^2t}\sfm_{12}\\
       0 & \e^{-(\nu-\min\{\nu,\kappa\})|\eta,l|^2t}  & -\e^{\min\{\nu,\kappa\}|\eta,l|^2t}\eta \sfm_{12} & \e^{\min\{\nu,\kappa\}|\eta,l|^2t}\eta\sfm_{22}\\ 
       0 & 0 & \e^{-\sfa t}\phi_-(\sfc t)	& 	- \e^{-\sfa t}\frac{\sfb}{\sfc}\sinh(\sfc t) \\
       0 & 0 & \e^{-\sfa t} \frac{\sfb}{\sfc}\sinh(\sfc t)&	\e^{-\sfa t}\phi_+(\sfc t)
    \end{pmatrix}.
\end{equation}
To bound each entry, notice that 
\begin{equation}\label{eq:stimecoshsinh1}
    \e^{-\sfa t}\cosh(\sfc t) \leq 2, \qquad \e^{-\sfa t}\frac{\sinh(\sfc t)}{\sfc}\leq \frac{2}{\max\{\sfa,\sfb\}}
\end{equation}
hold for any $\sfa, \sfb>0$, with $\sfc$ as in \eqref{eq:abc}.
Hence, 
\begin{align}
    |\e^{\min\{\nu,\kappa\}|\eta,l|^2t}\sfm_{11}|&=\left|\frac{\nu -\kappa}{\sfb^2}|\eta,l|^2 \left( \e^{-\sfa t} \phi_-(\sfc t)+\e^{-(\nu-\min\{\nu,\kappa\}) |\eta,l|^2t }\right)
    +\frac{\e^{-\sfa t}}{\sfc} \sinh(\sfc t)\right|\lesssim \frac{\sfa}{\sfb^2} +\frac{1}{\sfb},\\
    |\e^{\min\{\nu,\kappa\}|\eta,l|^2t}\sfm_{12}|&=\left|\frac{1}{\sfb}\left(\e^{-\sfa t}\phi_-(\sfc t)-\e^{-(\nu-\min\{\nu,\kappa\}|\eta,l|^2 t} \right)\right|\lesssim \frac{1}{\sfb},\\
    |\e^{\min\{\nu,\kappa\}|\eta,l|^2t}\sfm_{22}|&=\left|\frac{1}{\sfc} \e^{-\sfa t}\sinh(\sfc t)\right|\lesssim\frac{1}{\sfb}.\label{eq:stimecoshsinh2}
\end{align}
while all the other entries are bounded by $1$.
Now, recall that the solution to the ODE \eqref{eq:k=0modeFou} can be written explicitly via \eqref{eq:zeromodesolution}. Component-wise it reads
\begin{align}
        f_0(t)&=\e^{-\nu |\eta,l|^2 t}f_0(0) + \sfm_{11}g_0(0) + \sfm_{12}\gamma_0(0),\\
        h_0(t)&=\e^{-\nu |\eta,l|^2 t}h_0(0) - \eta\sfm_{12}g_0(0)+\eta\sfm_{22}\gamma_0(0),\\
        g_0(t)&=\e^{-\frac{\nu+\kappa}{2}|\eta,l|^2 t}\phi_-(\sfc t)g_0(0) -\e^{-\frac{\nu+\kappa}{2}|\eta,l|^2 t}\frac{\sfb}{\sfc}\sinh(\sfc t)\gamma_0(0),\\
        \gamma_0(t)&=\e^{-\frac{\nu+\kappa}{2}|\eta,l|^2 t}\frac{\sfb}{\sfc}\sinh(\sfc t)g_0(0) +\e^{-\frac{\nu+\kappa}{2}|\eta,l|^2 t}\phi_+(\sfc t)\gamma_0(0).
\end{align}
Using the estimates \eqref{eq:stimecoshsinh1} - \eqref{eq:stimecoshsinh2} we have the following point-wise bounds for the solutions 
\begin{align*}   
    |f_0(t)|&\lesssim \e^{-\min\{\nu,\kappa\} |\eta,l|^2 t}\left[|f_0(0)| + \frac{|\eta,l|^4}{\beta |l|^2}|g_0(0)| + \frac{|\eta,l|}{\beta |l|}(|g_0(0)|+|\gamma_0(0)|)\right],\\
    |h_0(t)|&\lesssim \e^{-\min\{\nu,\kappa\} |\eta,l|^2 t}\left[|h_0(0)| +\frac{|\eta||\eta,l|}{\beta |l|}|g_0(0)| + \frac{|\eta||\eta,l|}{\beta |l|}|\gamma_0(0)|\right], \\
    |g_0(t)|&\lesssim \e^{-\min\{\nu,\kappa\} |\eta,l|^2 t}\left[|g_0(0)| +|\gamma_0(0)|\right],\\
    |\gamma_0(t)|&\lesssim \e^{-\min\{\nu,\kappa\} |\eta,l|^2 t}\left[|g_0(0)| +|\gamma_0(0)|\right].
\end{align*}
By \eqref{eq:SimVarZero} and \eqref{eq:SimVarZero2}, we can translate these bounds into the originals variables $u^{1}_0,u^{2}_0,u^{3}_0,\theta_0$:
\begin{align*}
    |u^{1}_0(t)|&\lesssim \e^{-\min\{\nu,\kappa\} |\eta,l|^2 t}\left[|u^{1}_0(0)| + \frac{|\eta,l|^4}{\beta |l|^2}|u^{2}_0(0)| + \frac{1}{\beta}(|u^{2}_0(0)|+|\theta_0(0)|)\right],\\
    |u^{3}_0(t)|&\lesssim \e^{-\min\{\nu,\kappa\} |\eta,l|^2 t}\left[|u^{3}_0(0)| +\frac{|\eta|}{|l|}|u^{2}_0(0)| +\frac{|\eta|}{|\eta,l|}|\theta_0(0)|\right], \\
    |u^{2}_0(t)|&\lesssim \e^{-\min\{\nu,\kappa\} |\eta,l|^2 t}\left[|u^{2}_0(0)| +\frac{|l|}{|\eta,l|}|\theta_0(0)|\right],\\
    |\theta_0(t)|&\lesssim \e^{-\min\{\nu,\kappa\} |\eta,l|^2 t}\left[\frac{|\eta,l|}{|l|}|u^{2}_0(0)| +|\theta_0(0)|\right].
\end{align*}
Theorem \ref{thm:noliftup} follows immediately.

%%%%%%%%%%%%%%%%%%%%%%%%%%%%%%%%

 \section*{Acknowledgments} 
The research of MCZ was partially supported by the Royal Society URF\textbackslash R1\textbackslash 191492 and EPSRC Horizon Europe Guarantee EP/X020886/1.

\bibliographystyle{abbrv}
\bibliography{CZDZ-Linear3DBoussinesq.bib}

\end{document}